\documentclass[english,12pt]{article}
\usepackage[T1]{fontenc}
\usepackage[latin9]{inputenc}
\usepackage{verbatim}
\usepackage{amsmath, amsthm, amssymb}

\usepackage[numbers,sort&compress]{natbib}

\makeatletter
\theoremstyle{plain}
\newtheorem{thm}{\protect\theoremname}
\theoremstyle{plain}
\newtheorem{cor}{\protect\corollaryname}
\theoremstyle{plain}
\newtheorem{prop}{\protect\propositionname}
\theoremstyle{plain}
\newtheorem{lem}{\protect\lemmaname}

\usepackage[serif]{macros4tex}
\usepackage[unicode=true,
            pdfusetitle,
            bookmarks=true,
            bookmarksnumbered=false,
            bookmarksopen=false,
            breaklinks=false,
            pdfborder={0 0 0},
            pdfborderstyle={},
            backref=page, 
            colorlinks=true]
 {hyperref}
 \hypersetup{
    urlcolor= red!80!green!80!white, 
    citecolor= blue!80!green!80!white, 
    linkcolor= red!80!green!80!white}

\usepackage{geometry}
\makeatother

\usepackage{babel}
\providecommand{\corollaryname}{Corollary}
\providecommand{\lemmaname}{Lemma}
\providecommand{\propositionname}{Proposition}
\providecommand{\theoremname}{Theorem}

\begin{document}

\title{Nearly Optimal Robust Mean Estimation\\
	via Empirical Characteristic Function\thanks{This work is supported in part by Semiconductor Research Corporation (SRC) and DARPA.}}
\author{Sohail Bahmani\footnote{School of Electrical \& Computer Engineering, Georgia Institute of Technology, Atlanta, GA 30332.}\\
	\href{mailto:sohail.bahmani@ece.gatech.edu}{\texttt{sohail.bahmani@ece.gatech.edu}}}
\maketitle
\begin{abstract}
	We propose an estimator for the mean of random variables in separable real Banach spaces using the empirical characteristic function. Assuming that the covariance operator of the random variable is bounded in a precise sense, we show that the proposed estimator achieves the optimal sub-Gaussian rate, except for a faster decaying mean-dependent term. Under a mild condition, an iterative refinement procedure can essentially eliminate the mean-dependent term and provide a shift-equivariant estimate. Our results particularly suggests that a certain Gaussian width that appears in the best known rate in the literature might not be necessary. Furthermore, using the boundedness of the characteristic functions, we also show that, except possibly at high signal-to-noise ratios, the proposed estimator is gracefully robust against adversarial ``contamination''. Our analysis is overall concise and transparent, thanks to the tractability of the characteristic functions.
\end{abstract}

\section{Introduction}
Let $X,X_{1},\dotsc,X_{n}$ be independent and identically distributed
(i.i.d) random variables in a separable real Banach space $\left(\mbb B,\norm{\cdot}\right)$,
whose common law is $\mathcal{P}$. Let $\left(\mbb B^{*},\norm{\cdot}_{*}\right)$
denote the dual space for $\left(\mbb B,\norm{\cdot}\right)$. Furthermore,
denote the mean and the covariance operator for $\mc P$ respectively
by $\mu_{\star}$ and $\tg{\varSigma}\,:\,\mbb B^{*}\to\mbb B$, i.e.,
\begin{align*}
	\tg{\mu} & =\E\left(X\right)\,,
\end{align*}
and
\begin{align*}
	\tg{\varSigma}w & =\mathbb{E}\left(w\left(X-\tg{\mu}\right)\left(X-\tg{\mu}\right)\right)\,, & \text{for all linear functionals}\ w\in\mbb B^{*}\,.
\end{align*}

Our goal is accurate estimation of the mean of $\mc P$ (i.e., $\tg{\mu}$)
from the observed samples $X_{1},\dotsc,X_{n}$, only assuming that
the covariance operator $\tg{\varSigma}$ is \emph{nuclear}. For our purposes, the covariance operator $\tg{\varSigma}$ is nuclear,
if there exists a sequence $\left(z_{i}\right)_i$ in $\mbb B$ such
that $\tg{\varSigma}w=\sum_{i=1}^{\infty}w\left(z_{i}\right)z_{i}$
for all $w\in\mbb B^{*}$, and $\sum_{i=1}^{n}\norm{z_{i}}^{2}<\infty$; for a fully rigorous treatment, we refer the interested readers to \citep[pp. 107, and Ch. III]{VTC87}. In particular, the operator norm of $\tg{\varSigma}$, defined in
the usual sense as
\begin{align*}
	\norm{\tg{\varSigma}}_{\mr{op}} & =\sup_{w\,:\,\norm{w}_{*}\le1}\norm{\tg{\varSigma}w}\,,
\end{align*}
is assumed to be finite.

If $X$ is Gaussian, it follows from the concentration inequality for suprema of Gaussian processes (see, e.g., \citep[Theorem 5.8]{BLM13}) that the sample mean $\hat{\mu}=\sum_{i=1}^n X_i/n$ achieves the rate
\begin{align*}
	\norm{\hat{\mu} - \mu_\*} & \le \E\norm{\frac{1}{n}\sum_{i=1}^n X_i - \mu_\*} + \sqrt{\frac{2\norm{\varSigma_\*}_\mr{op}\log(1/\delta)}{n}}\,.
\end{align*}
The main question is whether there is an estimator $\hat{\mu}$ that attains the above \emph{purely sub-Gaussian rate}, up to constant factors, in a more general setting where $X$ is only known to have a finite second moment. Using the standard \emph{symmetrization} lemma (see, e.g., \citep[Lemma 6.4.2]{Ver18} or \citep[Lemma 2.3.6]{vDVW96}), the desired purely sub-Gaussian rate can be equivalently expressed in terms of a certain \emph{Rademacher complexity} denoted by
\begin{align}
	C_{n} & \defeq\frac{1}{\sqrt{n}}\,\E\left(\norm{\sum_{i=1}^{n}\varepsilon_{i}\left(X_{i}-\tg{\mu}\right)}\right)\,,\label{eq:Rademacher-complexity}
\end{align}
where $\varepsilon_{1},\dotsc,\varepsilon_{n}\in\left\{ \pm1\right\} $
are i.i.d. Rademacher random variables independent of the other random
variables. Specifically, an ideal estimator $\hat{\mu}$ would satisfy
\begin{align}
	\norm{\hat{\mu} -\mu_\*} & \lesssim \frac{C_n+\sqrt{\norm{\varSigma_\*}_\mr{op}\log(1/\delta)}}{\sqrt{n}}\,,\label{eq:ideal}
\end{align}
where here and throughout $A\lesssim B$ and $A\gtrsim B$, respectively
denote $cA\le B$ and $A\ge cB$ with $c>0$ being a constant.\footnote{We may also use the variants $\lesssim_q$ and $\gtrsim_q$ to signify the dependence of the hidden constant factor on a parameter $q$.}

It is well-known that the (raw) moments of a random variable, should
they exist, can be obtained from the derivatives of the corresponding
\emph{characteristic function} at the origin. In particular, the gradient
of the characteristic function at the origin is equal to the mean
of the random variable multiplied by the imaginary unit $\imath$.
Therefore, it is natural to consider mean estimators based on the
\emph{empirical characteristic function}. The empirical characteristic
function has a sharp concentration around the true characteristic
function, as it is an average of i.i.d. bounded and Lipschitz-continuous
functions. However, the gradient of the empirical characteristic function
need not be well concentrated. Inspired by the multivariate mean estimator of \citet{LM19c} that implicitly
uses convex conjugate functions to formulate a multidimensional median-of-means
estimator, we propose a mean estimator using the empirical characteristic
function that does not involve explicit differentiation. We show that the
proposed estimator achieves the optimal sub-Gaussian rate up to a
mean-dependent term of higher order incurred due to the lack of \emph{shift-equivariance}. Furthermore, we show that, under certain conditions that the lack of shift-equivariance impose, the proposed
estimator has a natural robustness against adversarial modifications
of the samples. We also provide an iterative scheme that effectively makes the estimator shift-equivariant in the non-adversarial setting.

\section{The Mean Estimator and Its Properties}

Denote the characteristic function of $\mc P$ by $\varphi\,:\,\mbb B^{*}\to\mbb C$,
explicitly defined as
\begin{align*}
	\varphi\left(w\right) & \defeq\E\left(e^{\imath\,w\left(X\right)}\right)\,,
\end{align*}
for every linear functional $w\in\mbb B^{*}$. Similarly, denote
the empirical characteristic function by
\begin{align*}
	\varphi_{n}(w) & \defeq\frac{1}{n}\sum_{j=1}^{n}e^{\imath\,w\left(X_{j}\right)}\,.
\end{align*}
For a prescribed parameter $r_{n}>0$, that may depend on $n$, we
study the estimator $\hat{\mu}$, defined using the imaginary part of $\varphi_n(\cdot)$ as
\begin{align}
	\hat{\mu}\ \in\  & \argmin_{\mu}\ \sup_{w\st\norm{w}_{*}\le r_{n}}\ r_{n}^{-1}\Bigl|w\left(\mu\right)-\underbrace{\frac{1}{n}\sum_{i=1}^n\sin(w(X_i))}_{=\mr{Im}\left(\varphi_{n}\left(w\right)\right)}\Bigr|\,.\label{eq:mean-estimator}
\end{align}

The proposed estimator is reminiscent of the mean estimators studied in \citep{Cat12}, \citep{CG17}, and \citep{Min18} that use some specific forms of nonlinear \emph{influence function} to reduce the effect of the heavy-tailed distribution. Let us consider the simpler case of mean estimation over the real line (i.e., $(\mbb B,\norm{\cdot})=(\mbb R, |\cdot|)$), and contrast with the univariate mean estimator due to \citet{Cat12}. Catoni's univariate estimator with a parameter $\alpha>0$, is designed to be shift-equivariant and can be expressed as
\begin{align}
	\hat{\mu}_\mr{C} & = \argmin_{\mu} \frac{1}{n}\sum_{i=1}^n \ell(\alpha(X_i - \mu))\,,\label{eq:Catoni}
\end{align}
where $\ell\st\mbb R\to \mbb R$ is a \emph{convex} loss function that is implicitly defined using its derivative (i.e., the influence function), required to satisfy
\begin{align}
	-\log(1-t+t^2/2) \le \ell'(t) \le \log(1+t+t^2/2)\,,\label{eq:influence-func}
\end{align}
for all $t\in\mbb R$. Clearly, $\ell'(\cdot)$ is non-decreasing due to the convexity of $\ell(\cdot)$. Leveraging these bounds on the influence function, the exponential moments for the terms appearing in the stationarity condition of \eqref{eq:Catoni} are bounded and used in a PAC-Bayesian argument in \citep{Cat12} to show,  for an appropriate choice of the parameter $\alpha$, the optimal sub-Gaussian accuracy of $\hat{\mu}_{\mr C}$, that is,
\begin{align*}
	|\hat{\mu}_{\mr{C}} - \tg{\mu}| & \lesssim \sqrt{\frac{\var(X)\log(2/\delta)}{n}}\,,
\end{align*}
holds with probability at least $1-\delta$.

The mechanism behind \eqref{eq:mean-estimator} is quite different, as is evident from the lack of shift-equivariance and the fact that $\sin(\cdot)$ is not monotonic. As mentioned in the previous section, the true mean can be viewed as the derivative of $w\mapsto\mr{Im}(\varphi(w)) = \E\sin(w X)$ at $w =0$. Equivalently, we can say the true mean $\tg{\mu}$ determines the \emph{best linear approximation} of $\mr{Im}(\varphi(w))$ in an arbitrarily small neighborhood of the origin $w=0$, that is, for a sufficiently small $r\ge0$ we have
\begin{align*}
	\sup_{|w|\le r}|w \tg{\mu}-\E\sin(w X)| & \le \sup_{|w|\le r}|w \mu -\E\sin(w X)| + o(r)\,,
\end{align*}
for any $\mu \in \mbb R$. Our proposed estimator uses this principle in the finite sample setting to approximate $\tg{\mu}$, using the empirical process $\sum_{i=1}^n\sin(w X_i)/n$ instead of $\E\sin(w X)$. The (multivariate) mean estimator proposed by \citet{CG17} is also effectively using the same principle for a certain monotonic influence function satisfying \eqref{eq:influence-func}. It is straightforward to show that as $r\downarrow 0$, the estimate $\hat{\mu}$ converges to the empirical mean. Therefore, to avoid the poor performance of the empirical mean under heavy-tailed distributions, we need to choose $r=r_n>0$ carefully to trade-off between robustness and accuracy. The desired sub-Gaussian behavior, up to the inevitable side effect of not being shift-equivariant, can be expected because of the sharp uniform concentration of the $\sum_{i=1}^n\sin(w X_i)/n$ around $\E\sin(w X)$.

The $\sin(\cdot)$ function, though not monotonic, has some properties that are shared with the influence functions considered in \citep{Cat12,CG17,Min18} such as obeying \eqref{eq:influence-func}, and being Lipschitz-continuous and approximately linear near the origin. These properties are important in our analysis as well. However, the $\sin(\cdot)$ function, or more precisely the complex exponential function, has another distinctive property with an equally important implication in our analysis. Specifically, the deviation of the empirical characteristic function, $\varphi_n(\cdot)$, from the true characteristic function, $\varphi(\cdot)$, remains invariant if the samples are translated by the same but otherwise arbitrary vector. This simple property is crucial in Proposition \ref{prop:CF-concentration} to achieve concentration bounds that are \emph{shift-invariant}. The effect of this property on the overall accuracy is that we only incur a $\norm{\tg{\mu}}$-dependent term of order $n^{-2/3}$ due to the lack of shift-equivariance in \eqref{eq:mean-estimator}, whereas accuracy of the estimator in \citep{CG17}, that also lacks shift-equivariance, depends on the second raw moment and effectively suffers from a $\norm{\tg{\mu}}$-dependent term  decaying at a slower rate of $n^{-1/2}$.

We show in the following theorem that the proposed multivariate mean estimator \eqref{eq:mean-estimator}
achieves a near-optimal rate for a specific choice of $r_{n}$.

\begin{thm}
	Recall the definition of $C_n$ in \eqref{eq:Rademacher-complexity}. For any confidence level $\delta\in\left(0,1\right)$, and any choice of the parameter $r_n>0$ in \eqref{eq:mean-estimator}, which defines the estimator $\hat{\mu}$, we have
	\begin{equation}
		\begin{aligned}
			\norm{\hat{\mu}-\tg{\mu}} & \le\frac{24\,C_{n}}{\sqrt{n}}+\sqrt{\frac{8\norm{\tg{\varSigma}}_{\mr{op}}\log\left(1/\delta\right)}{n}}+\frac{16\log\left(1/\delta\right)}{3nr_{n}} \\
			                          & \hphantom{\le}+\frac{1}{3}r_{n}^{2}\norm{\tg{\mu}}^{3}+r_{n}\norm{\tg{\varSigma}}_{\mr{op}}\,,
		\end{aligned}\label{eq:master}
	\end{equation}
	with probability at least $1-\delta$.

	In particular, given a prescribed level of accuracy $\epsilon>0$ that obeys
	\begin{align}
		\epsilon & \ge\max\left\{ \frac{96\,C_{n}+12\sqrt{\norm{\tg{\varSigma}}_{\mr{op}}\log\left(1/\delta\right)}}{\sqrt{n}},9\left(\frac{\log\left(1/\delta\right)}{n}\right)^{2/3}\norm{\tg{\mu}}\right\} \,,\label{eq:accuracy-level}
	\end{align}
	with $r_{n}$ chosen as
	\begin{align*}
		r_{n} & =\frac{22\log\left(1/\delta\right)}{n\epsilon}\,,
	\end{align*}
	the estimator $\hat{\mu}$ meets the prescribed accuracy, that is,
	\begin{align*}
		\norm{\hat{\mu}-\tg{\mu}} & \le\epsilon\,,
	\end{align*}
	with probability at least $1-\delta$.\label{thm:main}
\end{thm}
A variation of Theorem \ref{thm:main} can also provide accuracy guarantees
for the situation where only moments of order $1+\tau$ with $\tau\in\left(0,1\right)$
are known to exist. To establish this variation, we only need to use
different approximations of $\sin(\cdot)$ to adapt the proof of
Theorem \ref{thm:main} to the weaker moment assumption. More specifically, Proposition \ref{prop:CF-concentration} can be
easily modified to depend only on the moments of order strictly less
than two, and Lemma \ref{lem:sin-linear-approximation} can be invoked
with other appropriate choices of the parameters $p$ and $q$. For the sake of simpler exposition, we do not pursue these details
in this paper.

As a concrete example, let us consider the special case where $\mbb B$
is a separable Hilbert space and $\norm{\cdot}$ is induced by the
corresponding inner product. Recall that the covariance operator $\tg{\varSigma}$
is assumed to be nuclear. Thus, it has a finite \emph{trace} defined
as
\begin{align*}
	\tr\left(\tg{\varSigma}\right) & =\sum_{i=1}^{\infty}\inp{\psi_{i},\tg{\varSigma}\psi_{i}}\,,
\end{align*}
for some orthonormal basis $\left\{ \psi_{i}\right\} _{i=1}^{\infty}$
of $\mbb B$. With this definition of the trace, and using the Cauchy-Schwarz
inequality, we can upper bound $C_{n}$ as
\begin{align*}
	C_{n} & \le\sqrt{\frac{1}{n}\E\left(\norm{\sum_{i=1}^{n}\varepsilon_{i}\left(X_{i}-\tg{\mu}\right)}^{2}\right)} \\
	      & =\sqrt{\frac{1}{n}\,n\E\left(\norm{X-\tg{\mu}}^{2}\right)}                                              \\
	      & =\sqrt{\tr\left(\tg{\varSigma}\right)}\,.
\end{align*}
Therefore, Theorem \ref{thm:main} shows that, with probability at
least $1-\delta$, the estimator \eqref{eq:mean-estimator} can effectively
achieve the error bound
\begin{align*}
	\norm{\hat{\mu}-\tg{\mu}} & \lesssim\max\left\{ \sqrt{\frac{\tr\left(\tg{\varSigma}\right)}{n}}+\sqrt{\frac{\norm{\tg{\varSigma}}_{\mr{op}}\log\left(1/\delta\right)}{n}},\left(\frac{\log\left(1/\delta\right)}{n}\right)^{2/3}\norm{\tg{\mu}}\right\} \,.
\end{align*}
The first argument of the maximum operator above is indeed the optimal
sub-Gaussian rate pertaining to the case of Hilbert spaces (see, e.g.,
\citep{LM19c}).

The setup we consider is essentially the same as the setup in \citep{LM19a}, where the \emph{median-of-means} approach of \citep{LM19c} is generalized for estimation of the mean in a normed space $\left(\mbb R^{d},\norm{\cdot}\right)$ for a general norm $\norm{\cdot}$. Specifically, given a prescribed accuracy parameter $\epsilon>0$ and a confidence level $\delta\in\left(0,1\right)$, the estimator $\hat{\mu}$
in \citep{LM19a} is shown to achieve $\norm{\hat{\mu}-\tg{\mu}}\le\epsilon$ with probability at least $1-\delta$, if
\begin{align}
	\epsilon & \gtrsim\frac{1}{\sqrt{n}}\max\left\{ C_{n},\E\norm{G}+\norm{\tg{\varSigma}}_{\mr{op}}\sqrt{\log\left(2/\delta\right)}\right\} \,,\label{eq:LM-bound}
\end{align}
where $C_{n}$ is defined as in \eqref{eq:Rademacher-complexity},
and $G\in\mbb R^{d}$ is a centered Gaussian vector whose covariance
matrix coincides with $\tg{\varSigma}$. It is also argued in \citep{LM19a}
that \eqref{eq:LM-bound} basically describes a nearly optimal accuracy level\footnote{The lower bound arguments in \citep{LM19a}, inspect the special case of an isotropic Gaussian distribution, and need to be adjusted for the infinite-dimensional setting where the covariance operator has to be nuclear.}, with the term $\E\norm{G}$, that can be interpreted as a certain \emph{Gaussian width}, being the source of any potential
suboptimality. It should be emphasized that the appearance of $\E\norm{G}$ does \emph{not} dramatically affect the final approximations. For instance, in the Euclidean setting, both $C_n$ and $\E\norm{G}$ are conveniently bounded from above by $\sqrt{\tr(\tg{\varSigma})}$; see, e.g., the arguments in \citep[proof of Lemma 2.2]{MZ20} in the context of covariance estimation). The achievable accuracy level of our proposed estimator
\eqref{eq:accuracy-level} has an undesirable dependence on $\norm{\tg{\mu}}$,
but it suggests that the term $\E\norm{G}$ in \eqref{eq:LM-bound} might
indeed be redundant. Furthermore, estimators based on the median-of-means framework, including the estimator proposed in \citep{LM19c}, lack robustness against adversarial manipulation as shown in \citep{RV19}. Below in Section \ref{ssec:adversarial-contamination} we show that our proposed estimator \eqref{eq:mean-estimator} can exhibit this notion of robustness, albeit under certain conditions inflicted by the lack of \emph{shift-equivariance}.

Another closely related estimator is studied by \citet{Min18}. Leveraging (approximate) symmetrization of the distribution by ``batch averaging'', as in the median-of-means estimators, together with tempering, as in the Catoni's estimator \citep{Cat12,CG17}, a class of mean
estimators is proposed in \citep{Min18} that approach the optimal estimation rates. Requiring only existence of the second moment, these estimators achieve a rate that essentially has the desired form \eqref{eq:ideal} except for an extra term expressed by an intricate function governing the convergence rate in the corresponding central limit theorem. This extra term enters the bound through an exchange of the average of each batch of the samples by a Gaussian with matching first and second moments, and approximating the difference in the expectations by integrating the difference of the tail probabilities. The simplified upper bounds of the rate provided in \citep{Min18} are, however, in terms of central moment of order $2+\tau$,
for some $\tau\in(0,1]$. The accuracy level established in \citep{Min18} is roughly
\begin{align*}
	\epsilon & = O\left(\frac{C_n}{\sqrt{n}}+\sqrt{\frac{\norm{\tg{\varSigma}}_{\mr{op}}\log(2/\delta)}{n}}
	+ \sqrt{\frac{\norm{\tg{\varSigma}}_{\mr{op}}}{n^{1+\tau}}}\right)\,,
\end{align*}
that is extended to the adversarial contamination setting as well.

Leveraging the generality of Banach space, our results can be adapted to covariance estimation as well. The approach will be almost identical to the approach of \citet{MZ20} who adapted the guarantees for the multivariate median-of-means estimator in normed spaces \citep{LM19b} for covariance estimation. We do not intend to do the detailed derivations here, but we discuss the general ideas in connection with the results of \citep{MZ20}. Consider i.i.d. samples $(Y_i)_i$ of a
\emph{symmetrically distributed} random vector $Y\in \mbb R^d$, and set $(X_i=Y_iY_i^\T)_i$  and $X=YY^\T$. Since $Y$ is symmetric, $\E Y=0$, and the covariance matrix of $Y$ is simply $\varSigma_Y \defeq \E\left(YY^\T\right) = \E X$. Of importance for us is a setting considered in \citep{MZ20}, where $Y$ is assumed to have a finite fourth moment, and the $L_4$--$L_2$ norm equivalence that requires
\begin{align*}
	\left(\E (\inp{v,Y}^4)\right)^{1/4} & \le L \left(\E (\inp{v,Y}^2)\right)^{1/2}\,,
\end{align*}
for all $v\in \mbb R^d$, and some absolute constant $L>0$. Denote the spectral norm (or Schatten $\infty$-norm) by $\norm{\cdot}_{S_\infty}$, and let $\mfk{r}$ be  the \emph{effective rank} of $\varSigma_Y$ defined as
\[
	\mfk{r} \defeq \frac{\tr(\varSigma_Y)}{\norm{\varSigma_Y}_{S_\infty}}\,.
\]
Using the truncated samples $(X^{(\beta)}_i \defeq X_i\bbone_{\{\norm{X_i}_{S_\infty}\le \beta^2\}})_{i=1}^n$ for an appropriate truncation parameter $\beta >0$, \citep{MZ20} constructs an estimator $\hat{\varSigma}_{\mr{MZ}}$ for $\varSigma_Y$ that for $n\gtrsim_L \mfk{r}\log\mfk{r} + \log(1/\delta)$, with probability at least $1-\delta$, satisfies
\begin{align}
	\norm{\hat{\varSigma}_{\mr{MZ}} - \varSigma_Y}_{S_\infty} & \lesssim_L \norm{\varSigma_Y}_{S_\infty}\sqrt{\frac{\mfk{r}\log(\mfk{r})}{n}} + R^{(\beta)}\sqrt{\frac{\log(1/\delta)}{n}}\,, \label{eq:MZ-bound}
\end{align}
where $R^{(\beta)}=\sup\left\{\E\left(\left(v^\T(X_1^{(\beta)} - \E X_1^{(\beta)})u\right)^2\right)\st \norm{u}_2\le 1,\norm{v}_2\le 1\right\}\lesssim_L \norm{\varSigma_Y}_{S_\infty}$. The error bound \eqref{eq:MZ-bound} is obtained by approximating the relevant Rademacher complexity and operator norm in the context of covariance estimation \citep[proof of Lemma 2.2]{MZ20}.
The same approximations can be used to extract a similar guarantee from Theorem \ref{thm:main} for our estimator \eqref{eq:mean-estimator} reinterpreted in the context of covariance estimation. Interestingly, our results can reproduce \eqref{eq:MZ-bound}, up to some factors of $\log(1/\delta)$, as the ``mean-dependent'' term in
\eqref{eq:accuracy-level} effectively reduces to $(\log(1/\delta)/n)^{2/3}\norm{\varSigma_Y}_{S_\infty}$, which is particularly negligible if $R^{(\beta)}\gtrsim \norm{\varSigma_Y}_{S_\infty}$. The iterative scheme described below in Section \ref{ssec:improve-shift-equivariance} eliminates the gap between these results entirely.

\subsection{Making the estimator oblivious to $\epsilon$}

Theorem \ref{thm:main} suggests that our proposed estimator requires
the knowledge of an accuracy level $\epsilon$ obeying \eqref{eq:accuracy-level}.
Below we describe an approach to eliminate this requirement. With
minor modifications, the same argument also applies to the mean estimation
under adversarial contamination considered below in Section \ref{ssec:adversarial-contamination}.

For $t>0$, define $g_{t}\st\mbb B\to\mbb R_{\ge0}$ as
\begin{align*}
	g_{t}\left(\mu\right) & \defeq\frac{n}{11\log\left(1/\delta\right)}\sup_{w\st\norm{w}_{*}\le\frac{22\log\left(1/\delta\right)}{nt}}w\left(\mu\right)-\mr{Im}\left(\varphi_{n}\left(w\right)\right)\,.
\end{align*}
Furthermore, for every $t>0$, we consider a certain sublevel set
of the function $g_{t}\left(\cdot\right)$ denoted by
\begin{align*}
	\mc M_{t} & =\left\{ \mu\in\mbb B\st g_{t}\left(\mu\right)\le1\right\} \,.
\end{align*}
Because $g_{t}\left(\cdot\right)$ is a convex and coercive function,
the corresponding sublevel sets $\mc M_{t}$ are compact and convex.
Furthermore, the fact that $g_{t}(\mu)$ is decreasing in $t>0$,
implies that the sets $\mc M_{t}$ are monotonic, i.e., for $s\le t$
we have
\begin{align*}
	\mc M_{s} & \subseteq\mc M_{t}\,.
\end{align*}
Furthermore, in view of Proposition \ref{prop:master-prop} below,
for every $t>0$ the diameter of $\mc M_{t}$ with respect to $\norm{\cdot}$
is at most $t$. Theorem \ref{thm:main} basically guarantees that,
with probability at least $1-\delta$, for any $\epsilon$ satisfying
\eqref{eq:accuracy-level} we have $\tg{\mu}\in\mc M_{\epsilon}\subseteq\mc M_{2\epsilon}$.
The fact that $\left|\mr{Im}\left(\varphi_{n}\left(w\right)\right)\right|\le1$
for all $w\in\mbb B^{*}$, implies that for every $t>0$ we have
\begin{align*}
	\mc M_{t} & \subseteq\left\{ \mu\st\left\lVert \mu\right\rVert \le\frac{t}{2}\left(\frac{n}{11\log\left(1/\delta\right)}+1\right)\right\} \,.
\end{align*}
Only the following two mutually exclusive scenarios may occur. In
the first scenario, for all $t>0$ we have $0\in\mc M_{t}$, which
is equivalent to having
\begin{align*}
	\sup_{w\in\mbb B^{*}}\mr{Im}\left(\varphi_{n}\left(w\right)\right) & \le\frac{11\log\left(1/\delta\right)}{n}\,.
\end{align*}
Therefore, we can choose the estimator $\hat{\mu}=0\in\mc M_{\epsilon}$
to achieve the error $\norm{\hat{\mu}-\tg{\mu}}\le\epsilon$. In the
second scenario, for some $s>0$ we should have $\mc M_{s}=\varnothing$.
Therefore, by monotonicity of the sets $\mc M_{t}$, we can choose
$\epsilon_{0}>0$ such that $\mc M_{\epsilon_{0}}\ne\varnothing$,
but $\mc M_{\epsilon_{0}/2}=\varnothing$. This $\epsilon_{0}$ would
satisfy $\epsilon_{0}\le2\epsilon$, thus any estimator $\hat{\mu}\in\mc M_{\epsilon_{0}}\subseteq\mc M_{2\epsilon}$
meets the error bound $\norm{\hat{\mu}-\tg{\mu}}\le2\epsilon$.

\subsection{Computational considerations}
As shown by Proposition \ref{prop:master-prop}, the estimator \eqref{eq:mean-estimator}
is effectively a convex program. However, this convex program, as
is, does not appear to admit a computationally tractable solver even
under the finite dimensional setting. Because \eqref{eq:mean-estimator}
does not involve polynomials, the sum-of-squares machinery, which
provided tractable approximations for some other mean estimators (e.g.,
the multivariate median-of-means \citep{Hop18,CFB19}), does not seem
to be useful either. Finding an exact or approximate numerical solver
for \eqref{eq:mean-estimator} or similar estimators that may implicitly use the characteristic function, in a finite dimensional setting is an interesting problem for future research. In particular, computationally efficient robust mean estimation, with respect to a tractable, but otherwise general, norm is still an open problem.

\subsection{Building a shift-equivariant estimator}\label{ssec:improve-shift-equivariance}

The nuisance term that depends on $\norm{\tg{\mu}}$ in \eqref{eq:accuracy-level-contamination} is mainly due
to the lack of shift-equivariance in $\hat{\mu}$. If $\norm{\tg{\mu}}\lesssim n^{1/6}\left(C_{n}+\sqrt{\norm{\tg{\varSigma}}_{\mr{op}}\log\left(1/\delta\right)}\right)/\log^{2/3}\left(1/\delta\right)$,
then the estimator still achieves the purely sub-Gaussian rate
\begin{align*}
	\norm{\hat{\mu}-\tg{\mu}} & \le\epsilon=O\left(\frac{C_{n}}{\sqrt{n}}+\sqrt{\frac{\norm{\tg{\varSigma}}_{\mr{op}}\log\left(1/\delta\right)}{n}}\right)\,.
\end{align*}
Otherwise, the $\norm{\tg{\mu}}$-dependent term in \eqref{eq:accuracy-level},
despite vanishing at the rate $n^{-2/3}=o(n^{-1/2})$, prevents $\hat{\mu}$
from achieving the desired purely sub-Gaussian behavior.

However, this drawback can be remedied to a great extent, if we use \eqref{eq:mean-estimator}
to refine another estimator iteratively by ``descending''
with respect to $\norm{\tg{\mu}}$. The following corollary shows that this recursive procedure can effectively produce estimators with the desired purely sub-Gaussian rate, under a minimal assumption on the confidence level.
\begin{cor}
	Suppose that we are given $\hat{\mu}^{(0)}$ as a crude estimator of $\tg{\mu}$. Furthermore, with $\epsilon^{(0)} \ge \norm{\hat{\mu}^{(0)}-\tg{\mu}}$, we are given the sequence $\left(\epsilon^{(k)}\right)_k$ that obeys
	\begin{align*}
		\epsilon^{(k)} & \ge\max\left\{ \frac{96\,C_{n}+12\sqrt{\norm{\tg{\varSigma}}_{\mr{op}}\log\left(1/\delta\right)}}{\sqrt{n}},9\left(\frac{\log\left(1/\delta\right)}{n}\right)^{2/3}\epsilon^{(k-1)}\right\} \,,
	\end{align*}
	and choose
	\begin{align*}
		r^{(k)}_{n} & =\frac{22\log\left(1/\delta\right)}{n\epsilon^{(k)}}\,.
	\end{align*}
	Let $\varphi_n^{(k)}(w)\defeq \exp(-\imath w(\hat{\mu}^{(k)}))\varphi_n(w)$ denote the empirical characteristic function with respect to the translated samples $(X_i-\mu^{(k)})_i$, where the corresponding estimators $\hat{\mu}^{(k)}$ are recursively defined as
	\begin{align*}
		\hat{\mu}^{(k)}\ \in\  & \argmin_{\mu}\ \sup_{w\st\norm{w}_{*}\le r^{(k)}_{n}}\ {\left(r^{(k)}_{n}\right)}^{-1}\left|\inp{w\st[,]\mu-\hat{\mu}^{\left(k-1\right)}}-\mr{Im}\left(\varphi_{n}^{(k-1)}\left(w\right)\right)\right|\,.
	\end{align*}
	Then, with probability at least $1-\delta$, simultaneously for all $k=1,2,\dotsc$, we have
	\begin{align}
		\norm{\hat{\mu}^{(k)}-\tg{\mu}} & \le \epsilon^{(k)}\,.\label{eq:recursive-estimate}
	\end{align}
	In particular, if
	\begin{align*}
		n & \ge \max\left\{\epsilon^{-2}{\left(96\,C_n+12\sqrt{\norm{\tg{\varSigma}}_{\mr{op}}\log(1/\delta)}\right)}^2,30\log(1/\delta) \right\}\,,
	\end{align*}
	we can choose the sequence $(\epsilon^{(k)})_k$ such that
	\begin{align}
		\norm{\hat{\mu}^{(k)}-\tg{\mu}} & \le \max\left\lbrace\epsilon,{\left(\frac{9}{10}\right)}^{2k/3}\epsilon^{(0)}\right\rbrace \label{eq:recursive-estimate-explicit}
	\end{align}
	with probability at least $1-\delta$.
	\label{cor:recursive-estimate}
\end{cor}
\begin{proof}
	The fact that $|\exp(-\imath w(x))\varphi_n(w)-\exp(-\imath w(x))\varphi(w)|=|\varphi_n(w)-\varphi(w)|$ for all $x\in \mbb B$, implies that Proposition \ref{prop:CF-concentration} also applies simultaneously to all translated samples $\left(X_i-x\right)_i$, without diminishing the confidence level. It then becomes evident from the proof of Theorem \ref{thm:main}, that the theorem also applies simultaneously to all translated samples. In particular, the bound \eqref{eq:recursive-estimate} follows immediately by applying Theorem \ref{thm:main} recursively.

	Under the prescribed sample complexity, we have $27\log(1/\delta)/n \le 9/10$, using which we can choose
	\begin{align*}
		\epsilon^{(j)} & = \max\left\lbrace \epsilon, \left(\frac{9}{10}\right)^{2/3}\epsilon^{(j-1)} \right\rbrace\,,
	\end{align*}
	to obtain \eqref{eq:recursive-estimate-explicit} as desired.
\end{proof}

As an example let us consider again the case where $\mbb B$ is a Hilbert space for which $C_n\le\sqrt{\tr(\tg{\varSigma})}$. Let $\hat{\mu}^{\left(0\right)}$
to be the \emph{geometric median-of-means} \citep{Min15} of $n$ samples,
which achieves the accuracy $\norm{\hat{\mu}^{\left(0\right)}-\tg{\mu}}\lesssim\sqrt{\tr\left(\tg{\varSigma}\right)\log\left(2/\delta\right)}n^{-1/2}$ with probability at least $1-\delta/2$. If $\delta \ge 2\exp(-n/30)$, then with only $n$ additional samples, the procedure described by Corollary \ref{cor:recursive-estimate} can produce estimates $\hat{\mu}^{\left(1\right)},\hat{\mu}^{\left(2\right)},\dotsc$, that satisfy
\begin{align*}
	\norm{\hat{\mu}^{\left(k\right)}-\tg{\mu}} &
	\lesssim\max\left\lbrace \frac{\sqrt{\tr(\tg{\varSigma})}+\sqrt{\norm{\tg{\varSigma}}_{\mr{op}}\log(2/\delta)}}{\sqrt{n}},\left(\frac{9}{10}\right)^{2k/3}\sqrt{\frac{\tr\left(\tg{\varSigma}\right)\log(2/\delta)}{n}}\right\rbrace\,,
\end{align*}
with probability at least $1-\delta$. These estimators are effectively shift-equivariant, and clearly has a better accuracy compared to the geometric median-of-means estimator. Furthermore, for $k\gtrsim\log\log(2/\delta)$ we can have $(9/10)^{(4k/3)}\log(2/\delta)\le 1$ and $\hat{\mu}^{(k)}$ achieves the desired purely sub-Gaussian rate.

\subsection{Robustness in adversarial settings}\label{ssec:adversarial-contamination}

It is natural to expect that the estimator \eqref{eq:mean-estimator}
exhibits some form of robustness, because it accesses the data samples
through a bounded function (i.e., the empirical characteristic function).
The following corollary provides a precise guarantee for robustness
of the estimator under  the \emph{strong contamination} model
(see, e.g., \citep[Definition 1.1]{DK19}).

\begin{cor}[Robustness against strong contamination]
	Given a constant $\eta\in\left[0,1/2\right)$, suppose that a malicious adversary replaces the clean samples $X_{1},\dotsc,X_{n}$
	by the ``contaminated'' samples $\tilde{X}_{1},\dotsc,\tilde{X}_{n}$,
	with the only constraint that we have $\tilde{X}_{i}\ne X_{i}$ for at most $\eta n$ indices $i=1,\dotsc,n$.
	Let $\tilde{\varphi}_{n}\left(w\right)$ denote the empirical characteristic
	function with respect to the contaminated samples, i.e.,
	\begin{align*}
		\tilde{\varphi}_{n}\left(w\right) & \defeq\frac{1}{n}\sum_{i=1}^{n}e^{\imath\,w\left(\tilde{X}_{i}\right)}\,,
	\end{align*}
	and consider the mean estimator
	\begin{align*}
		\tilde{\mu} & \in\argmin_{\mu}\ \sup_{w\st\norm{w}_{*}\le r_{n}}\ r_{n}^{-1}\left|w\left(\mu\right)-\mr{Im}\left(\tilde{\varphi}_{n}\left(w\right)\right)\right|\,.
	\end{align*}
	With $C_{n}$ defined as in \eqref{eq:Rademacher-complexity} and
	for a confidence level $\delta\in\left(0,1\right)$, let $\epsilon>0$
	denote a prescribed level of accuracy that obeys
	\begin{align}
		\epsilon & \ge\max\Bigg\{\frac{96\,C_{n}+12\sqrt{\norm{\tg{\varSigma}}_{\mr{op}}\log\left(1/\delta\right)}}{\sqrt{n}}+8\sqrt{\eta\norm{\tg{\varSigma}}_{\mr{op}}},\,\left(19\eta+\frac{26\log\left(1/\delta\right)}{n}\right)^{2/3}\norm{\tg{\mu}}\Bigg\}\,.
		\label{eq:accuracy-level-contamination}
	\end{align}
	Then, for $r_{n}$ chosen as
	\begin{align*}
		r_{n} & =\frac{16\eta}{\epsilon}+\frac{22\log\left(1/\delta\right)}{n\epsilon}\,,
	\end{align*}
	the estimator $\tilde{\mu}$
	satisfies
	\begin{align*}
		\norm{\tilde{\mu}-\tg{\mu}} & \le\epsilon\,,
	\end{align*}
	with probability at least $1-\delta$.
	\label{cor:contamination-robustness}
\end{cor}

We make a few remarks on the result of the corollary before providing a proof. It is conventional to assume that the adversarial contamination rate $\eta$ is an absolute constant, not vanishing as $n\to \infty$. Therefore, for a sufficiently large $n$, the terms in \eqref{eq:accuracy-level-contamination} that involve $\eta$ will become dominant  and the prescribed accuracy bound should effectively satisfy
\begin{align*}
	\epsilon & \ge c\max\left\lbrace \sqrt{\eta\norm{\tg{\varSigma}}_\mr{op}},\eta^{2/3}\norm{\tg{\mu}}\right\rbrace\,,
\end{align*}
for some absolute constant $c>1$. This reveals that in the adversarial setting the dependence on $\norm{\tg{\mu}}$ can be more problematic than in the non-adversarial setting. For instance, we basically need to require $\eta\le c^{-3/2}$; otherwise we will have $\epsilon > \norm{\tg{\mu}}$, which trivially can be attained instead by zero as the estimator. The non-vanishing $\eta$ also significantly diminishes the effectiveness of the schemes for improving shift-equivariance, described previously in Section \ref{ssec:improve-shift-equivariance}. Nevertheless, if we are \emph{not} operating in a high signal-to-noise ratio regime in the sense that
\begin{align*}
	\norm{\tg{\mu}} & \lesssim {\left(\eta+\frac{\log(1/\delta)}{n}\right)}^{-1/6}\sqrt{\norm{\tg{\varSigma}}_\mr{op}}\,,
\end{align*}
then \eqref{eq:accuracy-level-contamination} reduces to
\begin{align*}
	\epsilon & \gtrsim \frac{C_n+\sqrt{\norm{\tg{\varSigma}}_\mr{op}\log(1/\delta)}}{\sqrt{n}} + \sqrt{\eta \norm{\tg{\varSigma}}_\mr{op}}
\end{align*}
which, for instance, matches the accuracy of the multivariate \emph{trimmed-mean} estimator established in \citep{LM19} for the special case where $\mbb B = \mbb R^d$ is equipped with the standard Euclidean norm.

If $X$ has a lighter tail, the trimmed-mean estimator \citep{LM19} as well as the estimator proposed in \citep{Min18} are shown to achieve a level of accuracy with a better dependence on $\eta$. The trimmed-mean estimator achieves an error of order $\eta\sqrt{\log(1/\eta)\norm{\tg{\varSigma}}_{\mr{op}}}$ if $X$ is sub-Gaussian \citep{LM19}. Furthermore, \citep{Min18} provides a general accuracy guarantee under adversarial contamination \citep[Corollary 3.1]{Min18}. Specializing this general result to the case of mean estimation in Euclidean spaces, under the conditions that
\begin{align*}
	\kappa_{2+\tau} & \defeq\sup_{w\st\norm{w}\le1}\frac{\E\left(\left|\inp{w,X-\tg{\mu}}\right|^{2+\tau}\right)}{\left(\var\left|\inp{w,X-\tg{\mu}}\right|\right)^{1+\tau/2}}\,,
\end{align*}
is finite, and the contamination level $\eta$ is restricted as $\eta\in\left(1/n,1/2\right)$, the estimator of \citep{Min18} is shown to
achieve the error bound
\begin{align*}
	\norm{\hat{\mu}-\tg{\mu}} & \lesssim\sqrt{\frac{\tr\left(\tg{\varSigma}\right)}{n}}+\sqrt{\lambda_{\max}\left(\tg{\varSigma}\right)}\left(\sqrt{\frac{\log\left(2/\delta\right)}{n}}+\eta^{\left(1+\tau\right)/\left(2+\tau\right)}\kappa_{2+\tau}^{1/\left(2+\tau\right)}\right) \\
	                          & \hphantom{\lesssim}+\eta^{-1/\left(2+\tau\right)}\kappa_{2+\tau}^{1/\left(2+\tau\right)}\frac{\log\left(2/\delta\right)}{n}\,,
\end{align*}
with probability at least $1-\delta$. These improvements in the error bound are significant for smaller values of $\eta$, particularly for $\eta = o(1)$. Our analysis, however, does not provide these refinements on the dependence on $\eta$.

\begin{proof}[Proof of Corollary \ref{cor:contamination-robustness}]
	It follows from Proposition \ref{prop:master-prop}, stated and proved
	below in Section \ref{sec:Proofs}, that
	\begin{align*}
		\norm{\tilde{\mu}-\tg{\mu}} & \le2r_{n}^{-1}\sup_{w\st\norm{w}_{*}\le r_{n}}\left|w\left(\tg{\mu}\right)-\mr{Im}\left(\tilde{\varphi}_{n}\left(w\right)\right)\right|
	\end{align*}
	Furthermore, with $\varphi_{n}\left(w\right)$ being the empirical
	characteristic function with respect to the clean data, by the triangle
	inequality we have
	\begin{equation}
		\begin{aligned}
			\sup_{w\st\norm{w}_{*}\le r_{n}}\left|w\left(\tg{\mu}\right)-\mr{Im}\left(\tilde{\varphi}_{n}\left(w\right)\right)\right| & \le\sup_{w\st\norm{w}_{*}\le r_{n}}\left|w\left(\tg{\mu}\right)-\mr{Im}\left(\varphi_{n}\left(w\right)\right)\right|                                               \\
			                                                                                                                          & \hphantom{\le}+\sup_{w\st\norm{w}_{*}\le r_{n}}\left|\mr{Im}\left(\varphi_{n}\left(w\right)\right)-\mr{Im}\left(\tilde{\varphi}_{n}\left(w\right)\right)\right|\,.
		\end{aligned}
		\label{eq:remove-contamination}
	\end{equation}
	Clearly,
	\begin{align}
		\sup_{w\st\norm{w}_{*}\le r_{n}}\left|\mr{Im}\left(\varphi_{n}\left(w\right)\right)-\mr{Im}\left(\tilde{\varphi}_{n}\left(w\right)\right)\right| & =\sup_{w\st\norm{w}_{*}\le r_{n}}\left|\frac{1}{n}\sum_{i\st\tilde{X}_{i}\ne X_{i}}\sin\left(w\left(\tilde{X}_{i}\right)\right)-\sin\left(w\left(X_{i}\right)\right)\right|\nonumber \\
		                                                                                                                                                 & \le2\eta\,.\label{eq:contamination-bound}
	\end{align}
	As shown in the proof of Theorem \ref{thm:main} below
	in Section \ref{subsec:Proof-of-Theorem}, for any arbitrary $r_{n}>0$,
	with probability at least $1-\delta$ we have
	\begin{align*}
		      & 2r_{n}^{-1}\sup_{w\st\norm{w}_{*}\le r_{n}}\ \left|w(\tg{\mu})-\mr{Im}\left(\varphi_{n}\left(w\right)\right)\right|                                                                                                               \\
		\le\  & \frac{24\,C_{n}}{\sqrt{n}}+\sqrt{\frac{8\norm{\tg{\varSigma}}_{\mr{op}}\log\left(1/\delta\right)}{n}}+\frac{16\log\left(1/\delta\right)}{3nr_{n}}+\frac{1}{3}r_{n}^{2}\norm{\tg{\mu}}^{3}+r_{n}\norm{\tg{\varSigma}}_{\mr{op}}\,.
	\end{align*}
	Therefore, on the same event, we
	deduce from \eqref{eq:remove-contamination} and \eqref{eq:contamination-bound}
	that
	\begin{equation}
		\begin{aligned}\norm{\tilde{\mu}-\tg{\mu}} & \le\frac{24\,C_{n}}{\sqrt{n}}+\sqrt{\frac{8\norm{\tg{\varSigma}}_{\mr{op}}\log\left(1/\delta\right)}{n}}+r_{n}^{-1}\left(4\eta+\frac{16\log\left(1/\delta\right)}{3n}\right) \\
			                            & \hphantom{\le}+\frac{1}{3}r_{n}^{2}\norm{\tg{\mu}}^{3}+r_{n}\norm{\tg{\varSigma}}_{\mr{op}}\,.
		\end{aligned}
		\label{eq:unspecialized-accuracy-contamination}
	\end{equation}
	In view of \eqref{eq:accuracy-level-contamination}, and for the
	prescribed value for $r_{n}$, we have the inequalities
	\begin{align*}
		\frac{24\,C_{n}}{\sqrt{n}}+\sqrt{\frac{8\norm{\tg{\varSigma}}_{\mr{op}}\log\left(1/\delta\right)}{n}} & \le\frac{\epsilon}{4}\,,
	\end{align*}
	\begin{align*}
		r_{n}^{-1}\left(4\eta+\frac{16\log\left(1/\delta\right)}{3n}\right) & \le\frac{\epsilon}{4}\,,
	\end{align*}
	\begin{align*}
		\frac{1}{3}r_{n}^{2}\norm{\tg{\mu}}^{3} & \le\frac{1}{3}\left(16\eta+\frac{22\log\left(1/\delta\right)}{n}\right)^{2}\left(19\eta+\frac{26\log\left(1/\delta\right)}{n}\right)^{-2}\epsilon\le\frac{\epsilon}{4}\,,
	\end{align*}
	and
	\begin{align*}
		r_{n}\norm{\tg{\varSigma}}_{\mr{op}} & \le\frac{\epsilon}{4}\,,
	\end{align*}
	that together with \eqref{eq:unspecialized-accuracy-contamination}
	yields the desired bound.
\end{proof}

\subsection{Related work}

Estimation of the mean of random variables from finite samples is
a classic problem in statistics, which has recently resurrected with
two new perspectives on robustness of the estimators. The primary
concern in the first perspective is robustness to heavy-tailed distributions.
An ideal estimator replicates the optimal statistical rate achieved
by the sample mean in the case of sub-Gaussian distributions, for
distributions that might \emph{not} have moments of order greater
than two \citep{LO11,Cat12,HS14,Min15,CG17,Min18,LM19c,LM19a,LM19,DL19}.
The second perspective on mean estimation focuses on robustness against
outliers, perhaps contrived by an adversary \citep{LRV16,DKK+17,DKK+19,CDG19}.
Recent activity in this area originated from the computer science
community who studied various models of data contamination and set
a higher priority for computational efficiency. Here, we only summarize
some of the results in the recent literature that are most relevant
to our work, and refer the interested readers to survey papers \citep{LM19b}
and \citep{DK19} for historical notes and an in-depth review of the
literature.

A natural approach to gain robustness in estimation is \emph{tempering}
of the tail of the distribution through nonlinear transformation of
samples. This approach is used, for instance, in \citep{Cat12,CG17}
to construct computationally tractable estimators. Using PAC-Bayesian
arguments, these estimators are shown to achieve the optimal sub-Gaussian
rate except for the dependence on the raw second moment rather than
the centered moment in the multi- or infinite-dimensional setting.

The \emph{median-of-means} estimator \citep{NY83,JVV86,AMS96,LO11}
achieves the desired purely sub-Gaussian rate in the univariate case
(see, e.g., \citep{LM19c}). Several generalizations are studied in
the literature by defining appropriate notions of the median in a
multidimensional setting. \emph{Geometric median}, that generalizes
the variational characterization of the univariate median, is used
in \citep{HS14,Min15} to create a multivariate median-of-means estimator.
For example, for random variables in a separable Hilbert space, the
geometric median-of-means achieves the estimation rate of the order
$\sqrt{\tr\left(\tg{\varSigma}\right)\log\left(1/\delta\right)}n^{-1/2}$
(see, \citep[Corollary 4.1 and subsequent remarks]{Min15}). The first
multivariate mean estimator with purely sub-Gaussian behavior is proposed
by \citet{LM19c}. This estimator is also a multivariate median-of-means
estimator in which the median of a set of points is defined to aggregate,
in a certain way, the medians of the projections of the points onto
every possible direction. A computationally tractable \emph{semidefinite relaxation}
for this median-of-means estimator, that enjoys the same
sub-Gaussian statistical rate, is first proposed by \citet{Hop18} based
on ideas from the sum-of-squares hierarchies. Building upon similar
ideas, an iterative estimator is proposed and analyzed by \citep{CFB19}
that has a significantly lower computational cost. The spectral algorithm by \citep{LLVZ20} has further reduced the computational cost and attained a running time of $O(n^2 d)$.

In the adversarial setting, a robust mean estimator is proposed in \citep{CDG19} that operates
in $\mbb R^{d}$ and relies on a semidefinite program to attenuate
the effect of the outliers planted by an adversary. With at most $\eta n$
contaminated samples, this estimator has a runtime of $\tilde{O}\left(nd\right)/\mr{poly}\left(\eta\right)$
and guarantees $\norm{\hat{\mu}-\tg{\mu}}_2=O\left(\sqrt{\eta\lambda_{\max}\left(\tg{\varSigma}\right)}\right)$,
though only with a constant probability. Inspired by the ideas in
\citep{CFB19} and \citep{CDG19}, the mean estimator proposed in
\citep{DL19} has a nearly linear runtime as function of $nd$, achieves
the sub-Gaussian rate, and is robust to outliers, albeit for a vanishing
level of contamination.

A multivariate trimmed-mean estimator is studied in \citep{LM19}
and shown to achieve the optimal sub-Gaussian accuracy rate in the non-adversarial setting,
and be gracefully robust against
adversarial contamination of the samples. This estimator basically
aggregates all of the directional ``trimmed'' empirical means in which
the majority of the projected samples between two quantiles, determined
by the level of contamination, are kept intact while the rest are
saturated.

\section{\label{sec:Proofs}Proofs }

In this section we prove the main theorem and the propositions that
it depends on.

\subsection{\label{subsec:Proof-of-Theorem}Proof of Theorem \ref{thm:main}}

Theorem \ref{thm:main} relies on the following proposition, that
is also used implicitly as a crucial step in the argument of \citep{LM19c}.
A more explicit usage of this result also appeared in \citep[proof of Corollary 3.1]{Min18} and \citep{LLVZ20}.
Proposition \ref{prop:master-prop} below, provides a slightly more
general variant of the mentioned results. Proof of this proposition
is straightforward and is provided below in Section \ref{subsec:prop1-proof}
for completeness.

\begin{prop}
	\label{prop:master-prop}
	Given a function $f:\mbb B^{*}\to\mbb R$ define $f^{*}:\mbb B\to\mbb R$ as
	\begin{align*}
		f^{*}\left(\theta\right) & \defeq\sup_{w\st\left\lVert w\right\rVert _{*}\le1}\ \left|f\left(w\right)-w\left(\theta\right)\right|\,,
	\end{align*}
	and let
	\begin{align*}
		\theta_{\min} & \in\argmin_{\theta}\:f^{*}\left(\theta\right)\,.
	\end{align*}
	Then, for all $\theta\in\mbb B$ we have
	\begin{align}
		\left\lVert \theta_{\min}-\theta\right\rVert & \le2f^{*}\left(\theta\right)\,.\label{eq:master-bound}
	\end{align}
	Furthermore, if the function $f$ is odd, i.e.,
	$f\left(-w\right)=-f\left(w\right)$ for all $w\in\mbb B^{*}$, the
	function $f^{*}$ is the convex conjugate of the function $f$ over
	the canonical unit ball of $\mbb B^{*}$.
\end{prop}

It follows from Proposition \ref{prop:master-prop}, by choosing $f\left(w\right)=r_{n}^{-1}\mr{Im}\left(\varphi_{n}(r_{n}\,w)\right)$,
that
\begin{align*}
	\norm{\hat{\mu}-\tg{\mu}} & \le2r_{n}^{-1}\sup_{w\st\norm{w}_{*}\le r_{n}}\left|w\left(\tg{\mu}\right)-\mr{Im}\left(\varphi_{n}\left(w\right)\right)\right|\,.
\end{align*}
Then, by the triangle inequality we obtain
\begin{align}
	\norm{\hat{\mu}-\tg{\mu}} & \le2r_{n}^{-1}\sup_{w\st\norm{w}_{*}\le r_{n}}\left|\mr{Im}\left(\varphi_{n}\left(w\right)-\varphi\left(w\right)\right)\right|+2r_{n}^{-1}\sup_{w\st\norm{w}_{*}\le r_{n}}\left|w\left(\tg{\mu}\right)-\mr{Im}\left(\varphi\left(w\right)\right)\right|\,.\label{eq:decomposition}
\end{align}

To upper bound the first supremum on the right-hand side of \eqref{eq:decomposition},
we use the following proposition that is adapted from a more general
result in \citep{Bou02}. The proof is provided below in Section \ref{subsec:prop2-proof}
for completeness.

\begin{prop}[Uniform concentration of Empirical Characteristic Function]
	With $C_{n}$ defined by \eqref{eq:Rademacher-complexity}, we have
	\begin{align}
		\sup_{w\st \norm{w}\le r_{n}}\left|\varphi_{n}(w)-\varphi(w)\right| & \le r_{n}\left(\frac{12\,C_{n}}{\sqrt{n}}+\sqrt{\frac{2\,\norm{\tg{\varSigma}}_{\mr{op}}\log\left(1/\delta\right)}{n}}\right)+\frac{8\log\left(1/\delta\right)}{3n}\,.\label{eq:CF-simple-tail-bound}
	\end{align}
	with probability \textup{at least} $1-\delta$.\label{prop:CF-concentration}
\end{prop}

It is worth mentioning that the term in the bound \eqref{eq:CF-simple-tail-bound} that decays as $n^{-1}$ is conservative. Some refinements can be achieved through techniques such as \emph{peeling} (see, e.g., \citep[Section 13.7]{BLM13}) or \emph{generic chaining} \citep{Tal14} (see also \citep{Dir15} and \citep[Section 8.5]{Ver18}). However, these refinements either do not substantially affect the overall accuracy of the estimator, or lead to quantities such as the previously mentioned Gaussian width that we intended to avoid.

Applying Proposition \ref{prop:CF-concentration} to \eqref{eq:decomposition},
with probability at least $1-\delta$, we have
\begin{equation}
	\begin{aligned}\norm{\hat{\mu}-\tg{\mu}} & \le\frac{24\,C_{n}}{\sqrt{n}}+\sqrt{\frac{8\,\norm{\tg{\varSigma}}_{\mr{op}}\log\left(1/\delta\right)}{n}}+\frac{16r_{n}^{-1}\log\left(1/\delta\right)}{3n} \\
		                          & \phantom{\le}+2r_{n}^{-1}\sup_{w\st\norm{w}_{*}\le r_{n}}\left|w\left(\tg{\mu}\right)-\mr{Im}\left(\varphi\left(w\right)\right)\right|\,.
	\end{aligned}
	\label{eq:decomposition+}
\end{equation}
Therefore, it suffices to upper bound the remaining supremum in \eqref{eq:decomposition+}.
Using the triangle inequality we have
\begin{equation}
	\begin{aligned}                    & \sup_{w\st\norm{w}_{*}\le r_{n}}\left|w\left(\tg{\mu}\right)-\mr{Im}\left(\varphi\left(w\right)\right)\right|                                                                                                          \\
		 & \le\sup_{w\st\norm{w}_{*}\le r_{n}}\left|w\left(\tg{\mu}\right)-\sin\left(w\left(\tg{\mu}\right)\right)\right|+\left|\E\left(\sin\left(w\left(X\right)\right)-\sin\left(w\left(\tg{\mu}\right)\right)\right)\right|\,.
	\end{aligned}
	\label{eq:CF-convex-conjugate-bound}
\end{equation}
In the current setting that $X$ is assumed to have a finite second moment, the terms on the right-hand side of \eqref{eq:CF-convex-conjugate-bound} can be bounded simply using the Taylor approximation of the $\sin(\cdot)$ function. However, we use the following lemma, proved below in Section \ref{subsec:lem-proof}, that is convenient to use even if $X$ has a stricter moment condition.

\begin{lem}
	For every $\alpha,\beta\in\mbb R$, and $p,q\in\left[0,1\right]$
	with $p+q>0$, we have
	\begin{align*}
		\text{\ensuremath{\left|\sin\alpha-\sin\beta-\left(\alpha-\beta\right)\cos\beta\right|}} & \le\frac{\left|\alpha-\beta\right|^{p+q+1}}{2^{p+q-1}\left(p+q+1\right)}+\frac{q\left|\alpha-\beta\right|^{p+q}\left|\beta\right|}{2^{p+q-2}\left(p+q\right)}\,.
	\end{align*}
	\label{lem:sin-linear-approximation}
\end{lem}

Applying Lemma \ref{lem:sin-linear-approximation} with $\alpha=w\left(\tg{\mu}\right)$,
$\beta=0$, and $p=q=1$ provides the bound
\begin{align}
	\left|w\left(\tg{\mu}\right)-\sin\left(w\left(\tg{\mu}\right)\right)\right| & \le\frac{1}{6}\left|w\left(\tg{\mu}\right)\right|^{3}\,.\label{eq:linear-sin-diff}
\end{align}
Invoking Lemma \ref{lem:sin-linear-approximation} once more with
$\alpha=w\left(X\right)$, $\beta=w\left(\tg{\mu}\right)$,
$p=1$, and $q=0$, we also have
\begin{align}
	\left|\E\left(\sin\left(w\left(X\right)\right)-\sin\left(w\left(\tg{\mu}\right)\right)\right)\right| & =\left|\E\left(\sin\left(w\left(X\right)\right)-\sin\left(w\left(\tg{\mu}\right)\right)-\cos\left(w\left(\tg{\mu}\right)\right)w\left(X-\tg{\mu}\right)\right)\right|\nonumber \\
	                                                                                                     & \le\frac{1}{2}\E\left({(w\left(X-\tg{\mu}\right))}^2\right)\nonumber                                                                                                           \\
	                                                                                                     & \le\frac{1}{2}\norm{w}_{*}^{2}\norm{\tg{\varSigma}}_{\mr{op}}\,,\label{eq:sin-diff-expectation}
\end{align}
where the first line holds simply because $w(\cdot)$ is linear, thereby $\E\left(w(X-\tg{\mu})\right)=0$.
Collecting \eqref{eq:CF-convex-conjugate-bound}, \eqref{eq:linear-sin-diff},
and \eqref{eq:sin-diff-expectation} we deduce that
\begin{align*}
	\sup_{w\st\norm{w}_{*}\le r_{n}}\left|w\left(\tg{\mu}\right)-\mr{Im}\left(\varphi\left(w\right)\right)\right| & \le\sup_{w\st\norm{w}_{*}\le r_{n}}\frac{1}{6}\left|w\left(\tg{\mu}\right)\right|^{3}+\frac{1}{2}\norm{w}_{*}^{2}\norm{\tg{\varSigma}}_{\mr{op}} \\
	                                                                                                              & \le\frac{1}{6}r_{n}^{3}\norm{\tg{\mu}}^{3}+\frac{1}{2}r_{n}^{2}\norm{\tg{\varSigma}}_{\mr{op}}\,.
\end{align*}
Therefore, in view of \eqref{eq:decomposition+}, the desired bound \eqref{eq:master} holds.

Because of \eqref{eq:accuracy-level} and the choice of
\begin{align*}
	r_{n} & =\frac{22\log\left(1/\delta\right)}{n\epsilon}\,,
\end{align*}
we have
\begin{align*}
	\frac{24\,C_{n}}{\sqrt{n}}+\sqrt{\frac{8\norm{\tg{\varSigma}}_{\mr{op}}\log\left(1/\delta\right)}{n}} & \le\frac{\epsilon}{4}\,,
\end{align*}
\begin{align*}
	\frac{16\log\left(1/\delta\right)}{3nr_{n}} & \le\frac{\epsilon}{4}\,,
\end{align*}
\begin{align*}
	\frac{1}{3}r_{n}^{2}\norm{\tg{\mu}}^{3} & \le\frac{{22}^{2}\log^{2}\left(1/\delta\right)}{3n^{2}\epsilon^{2}}\cdot\frac{n^{2}\epsilon^{3}}{9^{3}\log^{2}\left(1/\delta\right)}\le\frac{\epsilon}{4}\,,
\end{align*}
and
\begin{align*}
	r_{n}\norm{\tg{\varSigma}}_{\mr{op}} & \le\frac{22\log\left(1/\delta\right)}{n\epsilon}\frac{n\epsilon^{2}}{12^{2}\log\left(1/\delta\right)}\le\frac{\epsilon}{4}\,,
\end{align*}
thereby the derived error bound reduces to
\begin{align*}
	\norm{\hat{\mu}-\tg{\mu}} & \le\epsilon\,.\qedhere
\end{align*}

\subsection{\label{subsec:prop1-proof}Proof of Proposition \ref{prop:master-prop}}

By definition
\begin{align*}
	f^{*}\left(\theta_{\min}\right) & \le f^{*}\left(\theta\right),
\end{align*}
for any $\theta\in\mathbb{B}$ which implies that
\begin{align*}
	f^{*}\left(\theta_{\min}\right)+f^{*}\left(\theta\right) & \le2f^{*}\left(\theta\right).
\end{align*}
Therefore, to prove \eqref{eq:master-bound} it suffices to show that
\begin{align*}
	f^{*}\left(\theta_{\min}\right)+f^{*}\left(\theta\right) & \ge\left\lVert \theta_{\min}-\theta\right\rVert \,.
\end{align*}
Using the definition of the function $f^{*}$, we can write
\begin{align*}
	f^{*}\left(\theta_{\min}\right)+f^{*}\left(\theta\right) & =\sup_{w\st\left\lVert w\right\rVert _{*}\le1}\left|f\left(w\right)-w\left(\theta_{\min}\right)\right|+\sup_{w\st\left\lVert w\right\rVert _{*}\le1}\left|f\left(w\right)-w\left(\theta\right)\right| \\
	                                                         & \ge\sup_{w\st\left\lVert w\right\rVert _{*}\le1}\left(\left|f\left(w\right)-w\left(\theta_{\min}\right)\right|+\left|f\left(w\right)-w\left(\theta\right)\right|\right)                               \\
	                                                         & \ge\sup_{w\st\left\lVert w\right\rVert _{*}\le1}\left|w\left(\theta\right)-w\left(\theta_{\min}\right)\right|                                                                                         \\
	                                                         & =\sup_{w\st\left\lVert w\right\rVert _{*}\le1}\left|w\left(\theta-\theta_{\min}\right)\right|                                                                                                         \\
	                                                         & =\left\lVert \theta-\theta_{\min}\right\rVert \,,
\end{align*}
which proves the desired inequality.

Because the canonical unit ball of $\mbb B^{*}$ is symmetric, if $f\left(-w\right)=-f\left(w\right)$ for
all $w\in\mbb B^{*}$, then we have
\begin{align*}
	f^{*}\left(\theta\right) & =\sup_{w\st\left\lVert w\right\rVert _{*}\le1}\left|f\left(w\right)-w\left(\theta\right)\right| \\
	                         & =\sup_{w\st\left\lVert w\right\rVert _{*}\le1}w\left(\theta\right)-f\left(w\right)\,,
\end{align*}
which is clearly the convex conjugate of the function $f$ over the
unit ball of $\mbb B^{*}$.

\subsection{\label{subsec:prop2-proof}Proof of Proposition \ref{prop:CF-concentration}}

We will use the uniform concentration bound due to Bousquet \citep{Bou02}
to prove the desired tail bound. Define
\begin{align*}
	Z_{n} & \defeq n\sup_{w\st\norm{w}_{*}\le r_{n}}\left|\varphi_{n}(w)-\varphi(w)\right|\,,
\end{align*}
and
\begin{align*}
	v_{n} & \defeq n r_{n}^{2}\norm{\tg{\varSigma}}_{\mr{op}}+4\E\left(Z_{n}\right)\,.
\end{align*}
Observe that the simple identity $\left|e^{-\imath w(\tg{\mu})}\varphi_n(w)-e^{-\imath w(\tg{\mu})}\varphi(w)\right| = \left|\varphi_n(w)-\varphi(w)\right|$, enables us operate on the centered samples $\left(X_i -\tg{\mu}\right)_i$ and express $Z_{n}$ equivalently as
\begin{align}
	Z_{n} & =\sup_{w\st\norm{w}_{*}\le r_{n}}\left|\sum_{i=1}^{n}e^{\imath\,w\left(X_{i}-\tg{\mu}\right)}-\E\left(e^{\imath\,w\left(X_{i}-\tg{\mu}\right)}\right)\right|\nonumber                                                  \\
	      & =\sup_{w,b\st\norm{w}_{*}\le r_{n}\st[,]\left|b\right|\le\pi}\left|\sum_{i=1}^{n}\sin\left(w\left(X_{i}-\tg{\mu}\right)+b\right)-\E\left(\sin\left(w\left(X_{i}-\tg{\mu}\right)+b\right)\right)\right|,\label{eq:ImZn}
\end{align}
where the second line follows from the identity $|z|=\sup_{|b|\le \pi}\,\mr{Im}\left(e^{\imath\, b}z\right)$.
Furthermore, with $X'$ denoting an independent copy of $X$, the
variance of the summands in the expression of $Z_{n}$ in \eqref{eq:ImZn} can be bounded
as
\begin{align*}
	 & \sup_{w,b\st\norm{w}_{*}\le r_{n}\st[,]\left|b\right|\le\pi}\var\left(\sin\left(w\left(X-\tg{\mu}\right)+b\right)\right)                                                                             \\
	 & =\sup_{w,b\st\norm{w}_{*}\le r_{n}\st[,]\left|b\right|\le\pi}\E\left(\left|\sin\left(w\left(X-\tg{\mu}\right)+b\right)-\E\left(\sin\left(w\left(X'-\tg{\mu}\right)+b\right)\right)\right|^{2}\right) \\
	 & \le \sup_{w,b\st\norm{w}_{*}\le r_{n}\st[,]\left|b\right|\le\pi}\frac{1}{2}\E\left(\left|\sin\left(w\left(X-\tg{\mu}\right)+b\right)-\sin\left(w\left(X'-\tg{\mu}\right)+b\right)\right|^{2}\right)  \\
	 & =\sup_{w\st\norm{w}_{*}\le r_{n}}2\E\left(\left|\sin\left(\frac{1}{2}w\left(X-X'\right)\right)\right|^{2}\right)                                                                                     \\
	 & \le\sup_{w\st\norm{w}_{*}\le r_{n}}\frac{1}{2}\E\left({\left|w\left(X-X'\right)\right|}^2\right)                                                                                                     \\
	 & =r_{n}^{2}\norm{\tg{\varSigma}}_{\mr{op}}\,.
\end{align*}
Therefore, recalling the definition of $v_{n}$ and noting that the
magnitudes of the summands in \eqref{eq:ImZn} are no more than $2$,
we can invoke \citep[Theorem 2.3]{Bou02} and show that
\begin{align}
	\P\left(Z_{n}>\E\left(Z_{n}\right)+\sqrt{2v_{n}\log\frac{1}{\delta}}+\frac{2}{3}\log\frac{1}{\delta}\right) & \le\delta\,.\label{eq:CF-tail-bound}
\end{align}

Furthermore, with $\varepsilon_{1},\dotsc,\varepsilon_{n}$ being
i.i.d. Rademacher random variables, and $C_{n}$ defined by \eqref{eq:Rademacher-complexity},
we have the chain of inequalities
\begingroup
\allowdisplaybreaks
\begin{align}
	\E\left(Z_{n}\right) & =\E\left(\sup_{w\st\norm{w}_{*}\le r_{n}}\left|\sum_{i=1}^{n}e^{\imath\,w\left(X_{i}-\tg{\mu}\right)}-\E\left(e^{\imath\,w\left(X_{i}-\tg{\mu}\right)}\right)\right|\right)\nonumber                         \\
	                     & \le\E\left(\sup_{w\st\norm{w}_{*}\le r_{n}}\left|\sum_{i=1}^{n}\cos\left(w\left(X_{i}-\tg{\mu}\right)\right)-\E\left(\cos\left(w\left(X_{i}-\tg{\mu}\right)\right)\right)\right|\right)\nonumber             \\
	                     & \hphantom{\le}+\E\left(\sup_{w\st\norm{w}_{*}\le r_{n}}\left|\sum_{i=1}^{n}\sin\left(w\left(X_{i}-\tg{\mu}\right)\right)-\E\left(\sin\left(w\left(X_{i}-\tg{\mu}\right)\right)\right)\right|\right)\nonumber \\
	                     & =\E\left(\sup_{w\st\norm{w}_{*}\le r_{n}}\left|\sum_{i=1}^{n}\cos\left(w\left(X_{i}-\tg{\mu}\right)\right)-1-\E\left(\cos\left(w\left(X_{i}-\tg{\mu}\right)\right)-1\right)\right|\right)\nonumber           \\
	                     & \hphantom{\le}+\E\left(\sup_{w\st\norm{w}_{*}\le r_{n}}\left|\sum_{i=1}^{n}\sin\left(w\left(X_{i}-\tg{\mu}\right)\right)-\E\left(\sin\left(w\left(X_{i}-\tg{\mu}\right)\right)\right)\right|\right)\nonumber \\
	                     & \le2\E\left(\sup_{w\st\norm{w}_{*}\le r_{n}}\left|\sum_{i=1}^{n}\varepsilon_{i}(\cos\left(w\left(X_{i}-\tg{\mu}\right)\right)-1)\right|\right) \nonumber                                                     \\
	                     & \hphantom{\le}+2\E\left(\sup_{w\st\norm{w}\le r_{n}}\sum_{i=1}^{n}\varepsilon_{i}\sin\left(w\left(X_{i}-\tg{\mu}\right)\right)\right)\nonumber                                                               \\
	                     & \le6\E\left(\sup_{w\st\norm{w}_{*}\le r_{n}}\sum_{i=1}^{n}\varepsilon_{i}w\left(X_{i}-\tg{\mu}\right)\right) =6r_{n}\sqrt{n}\,C_{n}\,,\label{eq:CF-expectation-bound}
\end{align}
\endgroup
where the standard Gin\'{e}-Zinn symmetrization \citep{GZ84}, \citep[Lemma 2.3.1]{vDVW96},
and the \emph{Rademacher contraction principle} \citep[Theorem 4.12]{LT13} are invoked to obtain the second and the third inequalities, respectively.

Putting \eqref{eq:CF-tail-bound}, \eqref{eq:CF-expectation-bound},
and the inequality
\begin{align*}
	 & \E\left(Z_{n}\right)+\sqrt{2nr_{n}^{2}\norm{\tg{\varSigma}}_{\mr{op}}\log\frac{1}{\delta}}+\sqrt{8\E\left(Z_{n}\right)\log\frac{1}{\delta}}+\frac{2}{3}\log\frac{1}{\delta} \\
	 & \le2\E\left(Z_{n}\right)+\sqrt{2nr_{n}^{2}\norm{\tg{\varSigma}}_{\mr{op}}\log\frac{1}{\delta}}+\frac{8}{3}\log\frac{1}{\delta}\,,
\end{align*}
together yields the tail bound \eqref{eq:CF-simple-tail-bound}.

\subsection{\label{subsec:lem-proof}Proof of Lemma \ref{lem:sin-linear-approximation}}

Since $\left|\sin\alpha-\sin\beta-\left(\alpha-\beta\right)\cos\beta\right|$
is an even function, without loss of generality, we may assume $\alpha\ge\beta$.
Then, we can write
\begin{align*}
	\left|\sin\alpha-\sin\beta-\left(\alpha-\beta\right)\cos\beta\right| & =\left|\int_{\beta}^{\alpha}\left(\cos u-\cos\beta\right)\ \d u\right|                                          \\
	                                                                     & =\left|\int_{\beta}^{\alpha}-2\sin\left(\frac{u-\beta}{2}\right)\sin\left(\frac{u+\beta}{2}\right)\ \d u\right| \\
	                                                                     & \le2\int_{\beta}^{\alpha}\left|\frac{u-\beta}{2}\right|^{p}\left|\frac{u+\beta}{2}\right|^{q}\ \d u\,,
\end{align*}
where the inequality follows from the fact that $\left|\sin t\right|\le\left|\sin t\right|^{r}\le\left|t\right|^{r}$
for all $t\in\mbb R$ and $r\in\left[0,1\right]$. Furthermore, by
concavity of the function $t\mapsto t^{q}$ over $t\ge0$, and the
triangle inequality we have
\begin{align*}
	\left|\frac{u+\beta}{2}\right|^{q} & \le\left|\frac{u-\beta}{2}\right|^{q}+q\left|\frac{u-\beta}{2}\right|^{q-1}\left(\left|\frac{u+\beta}{2}\right|-\left|\frac{u-\beta}{2}\right|\right) \\
	                                   & \le\left|\frac{u-\beta}{2}\right|^{q}+q\left|\frac{u-\beta}{2}\right|^{q-1}\left|\beta\right|\,.
\end{align*}
Recalling the assumption $p+q>0$, we conclude that
\begin{align*}
	\left|\sin\alpha-\sin\beta-\left(\alpha-\beta\right)\cos\beta\right| & \le2\int_{\beta}^{\alpha}\left|\frac{u-\beta}{2}\right|^{p}\left(\left|\frac{u-\beta}{2}\right|^{q}+q\left|\frac{u-\beta}{2}\right|^{q-1}\left|\beta\right|\right)\ \d u\,, \\
	                                                                     & =\frac{\left|\alpha-\beta\right|^{p+q+1}}{2^{p+q-1}\left(p+q+1\right)}+\frac{q\left|\alpha-\beta\right|^{p+q}\left|\beta\right|}{2^{p+q-2}\left(p+q\right)}\,.
\end{align*}

\subsection*{Acknowledgments}
We thank the anonymous reviewers for their careful and detailed feedback. We also thank Vladimir Koltchinskii for his helpful suggestions on an earlier version of this manuscript.

\bibliographystyle{abbrvnat}
\bibliography{refs}

\end{document}